\documentclass[a4paper,12pt,openright]{article} 
\usepackage[utf8]{inputenc} 
\usepackage{verbatim}
\usepackage[a4paper]{geometry}
\usepackage{longtable}
\usepackage{epsfig}   
\usepackage{bbm}
\usepackage[normalem]{ulem} 
\usepackage{enumitem}
\usepackage{dsfont}

\usepackage{tikz}
\usetikzlibrary{arrows,decorations.pathmorphing,backgrounds,positioning,fit,petri}

\usepackage[T1]{fontenc}    
\usepackage{lmodern}        
\usepackage{color} 
\usepackage{a4wide}
\usepackage{amsmath}
\usepackage{amssymb}
\usepackage{amsthm} 
\theoremstyle{definition}
\newtheorem{defin}{Definition}[section]
\newtheorem{bem}[defin]{Remark}

\theoremstyle{plain}

\newtheorem{prop}[defin]{Proposition}

\newtheorem{theo}[defin]{Theorem}
\numberwithin{equation}{section} 

\newcommand{\R}{\mathbb{R}}
\newcommand{\divergenz}{\text{div}\,}
\newcommand{\supp}{\text{supp }}
\newcommand{\dInt}{\ensuremath{\mathrm{d}}}
\par   

\definecolor{middlegray}{rgb}{0.3,0.3,0.3}
\definecolor{lightgray}{rgb}{0.7,0.7,0.7}
\definecolor{red}{rgb}{.8,0,0}
\definecolor{blue}{rgb}{0,0,0.5}
\definecolor{yac}{rgb}{0.5,0.5,0.1}

\DeclareMathOperator{\dist}{\mathrm{dist}}

\definecolor{lightgray}{rgb}{0.5, 0.5, 0.5}
\definecolor{verylightgray}{rgb}{0.92, 0.9, 0.88}

\newcommand{\bN}{\ensuremath{\mathbb N}}
 
\newcommand{\cont}{\ensuremath{\mathcal{C}}}

\newcommand{\abs}[1]{\ensuremath{\lvert #1 \rvert}}

\newcommand{\nrm}[1]{\ensuremath{\lVert #1 \rVert}}

\newcommand{\nnrm}[2]{\ensuremath{\lVert #1 \rVert_{#2}}}		
	
\newcommand{\sgnnrm}[2]{\ensuremath{\Big\lVert #1 \Big\rVert_{#2}}}

\newcommand{\rklam}[1]{\left(#1\right)}					
\newcommand{\grklam}[1]{\big(#1\big)}
\newcommand{\sgrklam}[1]{\Big(#1\Big)}
\newcommand{\ssgrklam}[1]{\bigg(#1\bigg)}
\newcommand{\rrklam}[1]{\Bigg(#1\Bigg)}	
					
\newcommand{\ggklam}[1]{\big\{#1\big\}}						
\newcommand{\sggklam}[1]{\Big\{#1\Big\}}
\newcommand{\ssggklam}[1]{\bigg\{#1\bigg\}}					
\newcommand{\rgklam}[1]{\Bigg\{#1\Bigg\}}
    					
\newcommand{\eklam}[1]{\left[#1\right]}

\DeclareFontFamily{U}{matha}{\hyphenchar\font45}
\DeclareFontShape{U}{matha}{m}{n}{
      <5> <6> <7> <8> <9> <10> gen * matha
      <10.95> matha10 <12> <14.4> <17.28> <20.74> <24.88> matha12
      }{}
\DeclareSymbolFont{matha}{U}{matha}{m}{n}
\DeclareFontSubstitution{U}{matha}{m}{n}
\DeclareFontFamily{U}{mathx}{\hyphenchar\font45}
\DeclareFontShape{U}{mathx}{m}{n}{
      <5> <6> <7> <8> <9> <10>
      <10.95> <12> <14.4> <17.28> <20.74> <24.88>
      mathx10
      }{}
\DeclareSymbolFont{mathx}{U}{mathx}{m}{n}
\DeclareFontSubstitution{U}{mathx}{m}{n}
\DeclareMathDelimiter{\vvvert}{0}{matha}{"7E}{mathx}{"17}

\definecolor{frank}{rgb}{0.2,0.2,1.0} 
\definecolor{petru}{rgb}{0.7,0.1,0.1} 

\newcommand{\cef}{\color{black}}

\newcommand{\pci}{\color{petru}}
\newcommand{\icp}{\color{black}}

\let\oldmarginpar\marginpar
\renewcommand\marginpar[1]{%
\-\oldmarginpar[\raggedleft\footnotesize{\color{magenta} #1}]%
{\raggedright\footnotesize{\color{magenta} #1}}}

\begin{document}

\title{Besov Regularity for the Stationary Navier-Stokes Equation on Bounded Lipschitz Domains\thanks{The first author has been supported by the Deutsche Forschungsgemeinschaft (DFG, grant DA~360/19-1). The work of the last two authors have been supported by the Deutsche Forschungsgemeinschaft (DFG, grant DA~360/20-1). The second author has been also partially supported by the Marsden Fund Council from Government funding, administered
by the Royal Society of New Zealand.}}
\author{Frank Eckhardt\thanks{Corresponding author. Email: feckhardt@mathematik.uni-marburg.de}, Petru A. Cioica-Licht, Stephan Dahlke}
\date{}
\maketitle

\begin{abstract}
\noindent We use the scale $B^s_{\tau}(L_\tau(\Omega))$, $1/\tau=s/d+1/2$, $s>0$, to study the regularity of the stationary Stokes equation on bounded Lipschitz domains $\Omega\subset\R^d$, $d\geq 3$, with connected boundary. The regularity in these Besov spaces determines the order of convergence of nonlinear approximation schemes. Our proofs rely on a combination of weighted Sobolev estimates and wavelet characterizations of Besov spaces. By using Banach's fixed point theorem, we extend this analysis to the stationary Navier-Stokes equation with suitable Reynolds number and data, respectively.
\end{abstract}


\smallskip


\section{Introduction}\label{Sec:introduction}

The Navier-Stokes equations provide  a mathematical model of the motion of a fluid and form the basis for the theory of computational fluid
dynamics. Due to their relevance, the Navier-Stokes equations have been very intensively studied over the centuries, hence the amount of literature is enormous and can clearly not be  discussed in detail  here. Let us just refer to \cite{CoFo1988, Gal2011,
 Te2000, Tri2015} for an overview. An analytic description
of the solution is only available in rare cases, so that numerical
schemes for the constructive approximation of the solutions are needed.
Once again, the deluge of literature cannot be comprehensively presented
here; let  us just refer to \cite{CaHuQu1988, GiRa1986, Te2000} and the references therein.

In this paper, we consider  an important special case: The incompressible, steady-state, viscous Navier-Stokes equation given by
\begin{equation}\label{eq:NavierStokes}
\begin{alignedat}{3}
-\Delta u +  \nu u\cdot(\nabla u)+ \nabla \pi & = f &&~\text{ on }
\Omega, \\
   \divergenz u & = 0 &&~\text{ on } \Omega,\\
   u & =  g &&~\text{ on } \partial \Omega;
\end{alignedat}\tag{NAST}
\end{equation}
$\Omega$ denotes a bounded Lipschitz domain in $\R^d$, $d\geq 3$, with connected boundary,  $u$ is the velocity field of the fluid, $\pi$ denotes  the pressure, $f$ is  the given body force, $g$ is a prescribed velocity field, and
$\nu >0$ denotes the Reynolds number that describes the viscosity of the
fluid.  We are concerned with the regularity analysis  of 
solutions to \eqref{eq:NavierStokes}.  In particular, we study the
smoothness of solutions in the specific scale
\begin{equation}\label{eq:scale}
B^s_{\tau}(L_\tau(\Omega)),\quad
\frac{1}{\tau}=\frac{s}{d}+\frac{1}{2}, \quad s>0,\tag{\textasteriskcentered}
\end{equation}
of  Besov spaces.

The motivation for our analysis can be described as
follows.  Also for the stationary Navier-Stokes equations
(\ref{eq:NavierStokes}), an analytic description of the solution is
usually not possible, so that once again efficient numerical schemes
are needed.  A first natural idea would be to employ classical
nonadaptive schemes.  These methods correspond to a uniform space
refinement strategy, i.e., the underlying degrees of freedom are
uniformly distributed     and do not depend on the shape of the unknown
solutions.
As a rule of thumb, the convergence order of such a nonadaptive, uniform
scheme depends on the regularity  of the object one wants to approximate as
measured in the classical Sobolev scale, see, e.g.,
\cite{DeVo1998,Hac1992}. If the boundary of the underlying domain and the data of the equation are smooth enough, then also sufficiently high Sobolev
smoothness of the solutions to \eqref{eq:NavierStokes} can be expected, since this is the case for the linearized equation \eqref{St}, see \cite{AmrGir91}. However, on a general Lipschitz domain, boundary singularities may occur that significantly  diminish
the Sobolev smoothness, so that the convergence order of uniform schemes
drops down. In these cases,  adaptive algorithms suggest themselves.
Essentially, adaptive algorithms are tricky updating strategies. Based
on an a posteriori error estimator, additional degrees of freedom are
only spent in regions where the numerical approximation is still far
away from the exact solution. Therefore, in each step, the current
distribution of the degrees of freedom depends strongly on the  unknown
solution. Although the idea of adaptivity seems to be
convincing,  one principle problem remains. Adaptive schemes are very
hard to design, to analyze, and to  implement. Therefore, a rigorous
mathematical foundation that indicates that adaptivity really pays is
highly desirable.  Our line of attack to give a reasonable answer is
based on the following observation. Given a dictionary
$\{\psi_{\lambda}\}_{\lambda \in {\mathcal J}}$ of functions, the best
we can expect  is   that an adaptive scheme realizes the convergence
order of best $N$-term approximation with respect to this dictionary. In
this sense, best $N$-term approximation serves as the benchmark for
adaptive algorithms. In best $N$-term approximation, one does not
approximate by linear spaces but by nonlinear manifolds of the form
$$S_N:=\sggklam{ g~|~g=\sum_{\lambda \in \Lambda} c_{\lambda} \psi_{\lambda}, ~
|\Lambda|=N,~c_\lambda\in\mathbb{R}},$$
 i.e., one collects all functions $g$ for which the expansion with respect
to $\{\psi_{\lambda}\}_{\lambda \in {\mathcal J}}$ has at most $N$
nonvanishing coefficients. In many cases, e.g., if the dictionary
consists of a wavelet basis, adaptive algorithms that indeed realize the
convergence order of best $N$-term wavelet approximation schemes are
known to exist, see, e.g., \cite{CohDahDeV2001,CohDahDeV2002}. These  relationships
in mind, the following question arises:  what is the order of
convergence of best $N$-term approximation, and is it higher than the
order of nonadaptive, uniform  schemes? For then, the development of adaptivity
would be completely justified.
It is well-known that in many settings, e.g.,  for the wavelet case,
the order of approximation (in $L_2$) that can be achieved by best
$N$-term approximation exactly depends on the smoothness of the object
under consideration in the so-called {\em adaptivity scale}
\eqref{eq:scale}, i.e,
$$u  \in B^s_{\tau}(L_\tau(\Omega)),\quad
\frac{1}{\tau}=\frac{s}{d}+\frac{1}{2}~ \Longleftrightarrow
~\sum_{N=1}^{\infty}[N^{s/d}\sigma_N(u)]^{\tau}\frac{1}{N}< \infty,
~\sigma_N(u):=\inf_{g\in S_N}\|u-g\|_{L_2},$$
see, e.g, \cite{DeVo1998, DeVoJaPo1992} as well as \cite{DaNoSi2006} for similar relationships for approximations with respect to other norms. 
Consequently, in order to decide the question whether adaptivity pays in the context of Navier-Stokes equations, a rigorous analysis of the regularity of the solutions 
to~\eqref{eq:NavierStokes} in the scale~\eqref{eq:scale} is needed. 
If this regularity is higher than the classical $L_2$-Sobolev smoothness of the solutions under consideration, then adaptivity pays in the sense that there is indeed the possibility that adaptive methods exhibit higher convergence rates than their uniform alternatives.

This paper consists of two parts. In the first part, we
study the  linear version of \eqref{eq:NavierStokes}, i.e., the stationary Stokes
problem

\begin{equation}\label{St}
\begin{alignedat}{3}
-\Delta u +  \nabla \pi & = f &&~\text{ on } \Omega, \\
   \divergenz u & = 0 &&~\text{ on } \Omega,\\
   u & =  g &&~\text{ on } \partial \Omega,
\end{alignedat}\tag{SP}
\end{equation}
on an arbitrary bounded Lipschitz domain $\Omega\subseteq\R^d$, 
$d\geq 3$, with connected boundary.

For this class of problems,  some positive results concerning
Besov regularity in the scale~\eqref{eq:scale} already exist. In \cite{Dah1999}, the Stokes equations on a
polygonal domain in ${\mathbb R}^2$ have been studied.  The proofs were
based on decompositions
of the solutions into regular and singular parts. In \cite{Eck2015},
these results have been generalized to polyhedral domains, whereat
specific Kondratiev spaces have been employed. For the case of general
Lipschitz domains we are interested in here,
first results have been derived by Mitrea and Wright in \cite{MitMat2012}.  Their
proofs are based on the well-known concept of layer potentials.  In
this paper, we improve the results of~\cite{MitMat2012} in the following
sense. 

Our analysis shows that, other than conjectured in \cite[p.~9]{MitMat2012}, the results for the Besov smoothness in the scale \eqref{eq:scale} obtained therein, are not sharp for higher dimension $d\geq 4$. Our proof technique is completely different to the one used in \cite{MitMat2012}.  We first show regularity results in weighted Sobolev
spaces, and then we prove that these spaces can be embedded into the
Besov spaces corresponding to the adaptivity scale
(\ref{eq:scale}), which gives the desired results.

The second  part of this paper is concerned with the Besov regularity
of solutions to~\eqref{eq:NavierStokes}.  To the best of our
knowledge, no  regularity result in the scale~\eqref{eq:scale} has been
obtained so far. We tackle this task  by re-writing~\eqref{eq:NavierStokes} as a fixed point problem. For semilinear elliptic partial differential equations,  this strategy has already been successfully applied in
\cite{DaSi2013}. Nevertheless, there is an important difference. In
\cite{DaSi2013}, the fixed point theorems were  directly applied to
the quasi-Banach spaces in the adaptivity scale~\eqref{eq:scale}, whereas
here we study the problem first in classical Besov spaces  which enables us to reduce everything to the case of the Stokes problem with modified right-hand side. Then the desired Besov regularity results in the scale~\eqref{eq:scale} follows from the corresponding results for the Stokes problem. This approach has the advantage that certain admissibility problems that arise in the context of quasi-Banach spaces can be avoided. Moreover, the application of the Banach fixed point theorem guarantees uniqueness of the solution in a suitable, small ball. To the best of our knowledge, the non-standard fixed point arguments used in~\cite{DaSi2013} only provide the existence of a solution. We show that by proceeding this way we indeed  obtain  the desired result, in the sense that also for~\eqref{eq:NavierStokes} the Besov regularity of the solutions is higher than the standard Sobolev smoothness, so that the use of adaptivity is again completely justified.

This paper is organized as follows. In Section 2, we fix notation and briefly recall some basic concepts that are used in the sequel, in particular, concerning
function spaces and their wavelet characterization. Section 3 is devoted to the
Stokes problem. First of all, in Section 3.1,  we discuss (\ref{St}) in
weighted Sobolev spaces. We generalize regularity results obtained by
\cite{BroShe95} for the homogeneous Stokes equations to the
inhomogeneous case.  Then, in Section 3.2, we
prove that the weighted Sobolev spaces under consideration intersected with classical (unweighted) Sobolev spaces can be embedded
into the Besov spaces from the adaptivity scale (\ref{eq:scale}) up to a certain smoothness $s$ that depends on the Sobolev and the weighted Sobolev regularity.
A combination of these two facts provides our Besov regularity result. In
Section 4, we discuss the stationary Navier-Stokes equations
(\ref{eq:NavierStokes}). We use Banach's fixed point theorem to reduce the problem to the Stokes case. Then, we apply the results derived in Section 3.

\section{Preliminaries}\label{Sec:preliminaries}

\subsection{Notations}\label{Sec:Notation} 

In this paper $G \subseteq \mathbb{R}^d$, $d\geq 2$, stands for an arbitrary (not necessarily bounded) domain. 
By $\mathcal{D}'(G)$ we denote the space of Schwartz distributions on $G$. 
For $\alpha=(\alpha_1,...,\alpha_d)\in\mathbb{N}_0^d$, we write $D^\alpha f:=\frac{\partial^{|\alpha| }f }{\partial^{\alpha_d}x_d\ldots\partial^{\alpha_1}x_1}$ for the corresponding derivative of $f\in \mathcal{D}'(G)$, where $\abs{\alpha}:=\alpha_1+\ldots+\alpha_d$; $D^0f:=f$. 
For $m\in\bN_0$, $\nabla^m f:=\{D^\alpha f : \abs{\alpha}=m\}$ is the set of all $m^\text{th}$ order derivatives of $f$ and is identified with an $\R^n$-valued distribution, $n= {{d+m-1}\choose{m}}$. $\nabla:=\nabla^1$  denotes the gradient and $\Delta:=\sum_{i=1}^d\frac{\partial^2}{\partial x_i\partial x_i}$ is the Laplace operator. 
For $p\in [1,\infty)$ and $m\in\mathbb{N}_0$, $W^m(L_p(G))$ is the classical Sobolev space consisting of all (equivalence classes of) measurable functions $f\colon G\to \R$ such that $\nnrm{f}{W^m(L_p(G))}:=\left(\sum_{|\alpha|\leq m}\nnrm{D^\alpha f}{L_p(G)}^p\right)^{1/p}$ is finite. 
For $p\in (1,\infty)$ and fractional $s\in (0,\infty)\setminus\bN$, we define the Sobolev space $W^s(L_p(G))$ to be the Besov space $B^s_{p}(L_p(G))$, as defined below (see Definition~\ref{def:Bes}).
We write $\mathring{W}^s(L_p(G))$ for the closure with respect to the Sobolev norm $\nnrm{\cdot}{W^{s}(L_p(G))}$  of the space $\mathcal{C}_0^\infty(G)$ of infinitely differentiable functions with compact support within $G$. 
For negative $s<0$, $W^{s}(L_p(G))$ is defined as the dual space of  $\mathring{W}^{-s}(L_{p'}(G))$, where $1/p+1/p'=1$.
If $p=2$ we use the common notations $H^s(G):=W^s(L_2(G))$ and $\mathring{H}^s(G):=\mathring{W}^s(L_2(G))$, $s\in\R$. 
By making slight abuse of notation, we sometimes use the same abbreviations for $\mathbb R^d$-valued (generalized) functions.
Moreover, we use the common notation
\[
u\cdot(\nabla v):=\sum_{i=1}^d u_i\frac{\partial v}{\partial x_i},
\]
for $d$-dimensional (generalized) functions $u$ and $v$ and write $\R_G$ for the set of all real-valued constant functions on a domain $G$.

%
Throughout, we denote by $\Omega$ a bounded Lipschitz domain contained in $\mathbb{R}^d$, $d\geq 3$,  which in some of the central statements is assumed to have connected boundary. 
We set
$$\rho(x):=\dist(x,\partial\Omega),\quad x\in\Omega, $$
and define the weighted Sobolev space $W^m_\alpha(L_p(\Omega))$ for $m\in\bN_0$,  $\alpha >0$ and $p\in[1,\infty)$ as
$$W^m_\alpha(L_p(\Omega)):=\left\{f\in L_p(\Omega): 
 \nnrm{f}{W^m_\alpha(L_p(\Omega))}^p
:=
\nnrm{f}{L_p(\Omega)}^p+\int_\Omega \rho(x)^\alpha \abs{\nabla^m f(x)}_{\ell_p}^p \dInt x<\infty \right\}, $$
where $\abs{\nabla^mf}_{\ell_p}$ is the $\ell_p$-norm of the vector $\nabla^m f$. 
For $p\in (0,\infty)$ and $\max\{0,(d-1)(1/p-1)\}<s<1$ we define the Besov spaces $B^s_p(L_p(\partial\Omega))$ as they were introduced in \cite[Chapter~2.5]{MitMat2012}. We further introduce the subspace
\[
 B^{s,0}_{p}(L_p(\partial\Omega)):=\ggklam{g\in B^{s}_p(L_p(\partial\Omega)): \int_{\partial\Omega}g\cdot\mathbf{n}\,\mathrm{d}\sigma=0},
\]
where $\mathbf{n}$ is the outward unit normal to $\partial\Omega$. We norm this space with the norm inherited from $B^{s}_p(L_p(\partial\Omega))$.
Further we put $W^s(L_p(\partial\Omega)):=B^s_p(L_p(\partial\Omega))$ for $p\in [1,\infty)$ and $s\in (0,1)$. The space $W^1(L_2(\partial\Omega))$ is defined analogously, see for example \cite{BroShe95}.  For $p=2$ we also write $H^s(\partial\Omega):=W^s(L_2(\partial\Omega))$, $s\in (0,1]$. 

For arbitrary normed spaces $E_1,...,E_n$, $n\in\bN$, the Cartesian product $E_1\times \ldots\times E_n$ is endowed with the norm 
$$
\nnrm{(e_1,...,e_n)}{E_1\times \ldots\times E_n}
:=
\sum_{i=1}^n \nnrm{e_i}{E_i},\quad (e_1,...,e_n)\in E_1\times \ldots\times E_n; $$ 
we write shorthand $E^n$ if $E_j=E$ for all $j=1,\ldots,n$.
The intersection of two normed spaces $(E_1,\nnrm{\cdot}{E_1})$ and $(E_2,\nnrm{\cdot}{E_2})$ is normed by
\[
\nnrm{f}{E_1\cap E_2}:=\nnrm{f}{E_1}+\nnrm{f}{E_2}, \qquad f\in E_1\cap E_2.
\]
If $E_1\subseteq E_2$ and there exists a constant $C\in (0,\infty)$, such that
\[
\nnrm{f}{E_2}\leq C\nnrm{v}{E_1},\quad v\in E_1,
\]
then we write $E_1\hookrightarrow E_2$ and say that $E_1$ is embedded in $E_2$.
 Quotient spaces $E/E_0:=\{x+E_0\colon x\in E\}$ of a normed space $(E,\nnrm{\cdot}{E})$ and a subspace $E_0\subseteq E$ are endowed with the usual norm
\[
\nnrm{f}{E/E_0}:=\inf_{g\in E_0}\nnrm{f+g}{E},
\]
where we make use of the common abuse of notation to write simply $f$ instead of the equivalence class $f+E_0$.
Throughout, the letter $C$ denotes a finite positive constant that may differ  from one appearance to another, even in the same chain of inequalities.

\subsection{Besov spaces and wavelet decompositions}\label{Sec:BesovWavelet} 

In this section we present the definition of Besov spaces and describe their wavelet characterization.
Our standard references in this context are  \cite{Coh2003}, \cite{Dis2003} and \cite{Tri1983}.

We introduce the Besov spaces $B^s_q(L_p(G))$ by using the common Fourier-analytical approach. Therefore, we fix an arbitrary function 
$\varphi_0\in \mathcal{C}_0^\infty(\mathbb{R}^d)$ with 
\begin{align*}
\varphi_0(x)=1 \text{ if } |x|\leq 1 \quad\text{and}\quad
\varphi_0(x)=0 \text{ if } |x|\geq 3/2,
\end{align*}
and define for $k\in\bN$,
\begin{align*}
\varphi_k(x) := \varphi_0(2^{-k}x)-\varphi_0(2^{-k+1}x) \quad \text{for} \quad x\in\mathbb{R}^d,
\end{align*}
to obtain a smooth dyadic resolution of unity on $\R^d$, i.e.,  $\varphi_k\in \mathcal{C}_0^\infty(\mathbb{R}^d)$ for all $k\in\mathbb{N}_0$, and
\begin{align*}
\sum_{k=0}^\infty \varphi_k(x) = 1 \quad\text{for all}\quad x\in\mathbb{R}^d.
\end{align*}
We write $\mathfrak{F}$ for the Fourier transform on the space $\mathcal{S}'(\mathbb{R}^d)$ of tempered distributions. Recall that $\mathfrak{F}^{-1}(\varphi_k \mathfrak{F}f)$ is an entire analytic function for arbitrary $f\in \mathcal{S}'(\mathbb{R}^d)$ and $k\in\bN_0$.

\begin{defin}\label{def:Bes}
Let $\{\varphi_k\}_{k\in\mathbb{N}_0}\subseteq \mathcal{C}_0^\infty(\mathbb{R}^d)$ be a resolution of unity as described above. Furthermore, let $0<p,q<\infty$ and $s\in\mathbb{R}$.

\medskip

\noindent\textbf{(i)} The \emph{Besov space $B^s_q(L_p(\R^d))$} is defined by 
\begin{align*}
B^s_q(L_p(\mathbb{R}^d)):=
\rgklam{
f\in \mathcal{S}'(\mathbb{R}^d)
\colon
\nnrm{f}{B^s_q(L_p(\mathbb{R}^d))} :=
\ssgrklam{\sum_{k=0}^\infty 2^{ksq}\nnrm{\mathfrak{F}^{-1}\eklam{\varphi_k\mathfrak{F}f}}{L_p(\R^d)}^q}^{1/q}
<\infty}.
\end{align*}

\smallskip

\noindent\textbf{(ii)} Let $G\subseteq\R^d$ be an arbitrary domain. Then, the \emph{Besov space $B^s_q(L_p(G))$} is defined by
$$
B^s_q(L_p(G)):=\ggklam{f\in \mathcal{D}'(G)\colon \text{there exists } g\in B^s_q(L_p(\mathbb{R}^d)): g|_G=f}. $$
It is equipped with the (quasi-)norm
\begin{equation*}
\nnrm{f}{B^s_q(L_p(G))}:= \inf\ggklam{\nnrm{g}{B^s_q(L_p(\mathbb{R}^d))}: g\in B^s_q(L_p(\mathbb{R}^d)), g|_G=f},\quad f\in B^s_q(L_p(G)).
\end{equation*}
\end{defin}

\begin{bem}
Besides the definition given above, Besov spaces $B^s_q(L_p(G))$ of positive smoothness $s>0$ are frequently defined by means of iterated differences, see for example \cite{Tri1983}. The two definitions coincide in the sense of equivalent norms for the range of parameters $s>\max\{0,d\cdot (1/p-1)\}$, provided, e.g., $G$ is a bounded Lipschitz domain or $G=\mathbb{R}^d$, see, e.g., \cite[Theorem~3.18]{Dis2003} and \cite[Theorem~2.5.12]{Tri1983}, respectively. 
\end{bem}

\vspace{0.1cm}

In order to present a characterization of Besov spaces on $\R^d$ in terms of wavelets, we fix the following setting. 
 Let $\phi$ be a scaling function of tensor product type on $\mathbb{R}^d$ and let $\psi_i$, $i=1, \ldots, 2^d-1$, be corresponding multivariate  mother  wavelets such that, for a given $r\in\mathbb{N}$ and some cube $Q$ centered at the origin, the following locality, smoothness and vanishing moment conditions hold. For all $i=1, \ldots, 2^d-1$,
\begin{align}
&\supp \phi,\,\supp\psi_i\subseteq Q,\label{wl1}\\
&\phi,\,\psi_i \in \mathcal{C}^r(\mathbb{R}^d),\label{wl2}\\
&\int_{\mathbb{R}^d} x^\alpha \, \psi_i (x)\, dx=0 \quad\text{ for all $\alpha\in\mathbb{N}_0^d$ with $|\alpha|\leq r$}\label{wl3}.
\end{align}
For the dyadic shifts and dilations of the scaling function and the corresponding wavelets we use the abbreviations
\label{shift}
\begin{align}
\phi_k(x)    &:=\phi(x - k),\;x\in\mathbb{R}^d ,&&\text{for $k\in\mathbb{Z}^d$\label{wl4}, and}\\
\psi_{i,j,k}(x) &:=2^{jd/2}\psi_i(2^jx-k),\;x\in\mathbb{R}^d, &&\text{for $(i,j,k)\in\{1,\ldots,2^d-1\}\times\mathbb{N}_0\times\mathbb{Z}^d$\label{wl5},}
\end{align}
and assume that
\begin{align*}
 \left\{\phi_k,\psi_{i,j,k}\,:\, (i,j,k)\in\{1,\ldots,2^d-1\}\times\mathbb{N}_0\times\mathbb{Z}^d\right\}
\end{align*}
is a Riesz basis of $L_2(\mathbb{R}^d)$.
Further, we assume that there exists a dual Riesz basis satisfying the same requirements. That is, there exist functions \label{dualRiesz}
$\widetilde{\phi} $ and
$\widetilde{\psi}_i$, $ i=1, \ldots , 2^d-1$,
 such that conditions \eqref{wl1}, \eqref{wl2} and \eqref{wl3} hold if $\phi$ and $\psi_i$ are replaced by $\widetilde{\phi} $ and
$\widetilde{\psi}_i$, and such that the biorthogonality relations
\begin{equation*}
\langle \widetilde{\phi}_k, \psi_{i,j,k} \rangle = \langle \widetilde{\psi}_{i,j,k},
 \phi_k \rangle  = 0\, , \quad
\langle \widetilde{\phi}_k, \phi_{l} \rangle  = \delta_{k,l}, \quad
\langle \widetilde{\psi}_{i,j,k}, \psi_{u,v,l} \rangle  = \delta_{i,u}\, \delta_{j,v}\, \delta_{k,l}\, ,
\end{equation*}
are fulfilled. Here we use  analogous  abbreviations to \eqref{wl4} and \eqref{wl5} for the dyadic shifts and dilations of $\widetilde{\phi} $ and
$\widetilde{\psi}_i$,  and $\delta_{k,l}$ denotes the Kronecker symbol. We refer to \cite[Chapter 2]{Coh2003} for the construction of biorthogonal wavelet bases, see also \cite{CohDauFea1992} and \cite{Dau1992}. 


Such a wavelet basis at hand, it is possible to characterize Besov spaces by the decay of the wavelet coefficients in the following way. A proof can be found in \cite[Theorem~3.7.7]{Coh2003}.

\begin{prop}\label{Prop:CharacterizationBesov1}
Let $p,q\in\left(0,\infty\right)$ and $s>\max\left\{0,d\left(1/p-1\right)\right\}$. Choose $r\in\mathbb{N}$ such that $r>s$ and construct a  biorthogonal  wavelet Riesz basis as described above.
Then a locally integrable function $f:\mathbb{R}^d\to\mathbb{R}$ is in the Besov space $B^s_{q}(L_p(\mathbb{R}^d))$ if, and only if,
\begin{equation*}
f
=
\sum_{k\in\mathbb{Z}^d}\langle f,\widetilde{\phi}_k\rangle\,\phi_k
+ 
\sum_{i=1}^{2^{d}-1}\sum_{j\in\mathbb{N}_0}\sum_{k\in\mathbb{Z}^d}\langle f,\widetilde{\psi}_{i,j,k}\rangle\,\psi_{i,j,k}
\end{equation*}
(convergence in $\mathcal D'(\mathbb{R}^d)$) with
\begin{equation}\label{BesovNormDiscrete}
\left(\sum_{k\in\mathbb{Z}^d}\left|\langle f,\widetilde{\phi}_k\rangle\right|^p
\right)^{\frac{1}{p}}
+ 
\left(
\sum_{i=1}^{2^{d}-1}\sum_{j\in\mathbb{N}_0}
2^{j\left( s+d\left(\frac{1}{2}-\frac{1}{p}\right)\right) q}
\left(\sum_{k\in\mathbb{Z}^d}\left|\langle f,\widetilde{\psi}_{i,j,k}\rangle\right|^p\right)^{\frac{q}{p}}\right)^{\frac{1}{q}}
< \infty,
\end{equation}
and \eqref{BesovNormDiscrete}  is an equivalent (quasi-)norm  for $B^s_{q}(L_p(\mathbb{R}^d))$.
\end{prop}

A short computation shows that the Besov spaces $B^s_{\tau}(L_\tau(\mathbb{R}^d))$, with $1/\tau=s/d+1/2$, $s>0$, admit the following characterization.

\begin{prop}\label{Prop:CharacterizationBesov2}
Let $s>0$ and $\tau\in\mathbb{R}$ such that  $1/\tau=s/d+1/2$. Choose $r\in\mathbb{N}$ such that $r>s$ and construct a  biorthogonal  wavelet Riesz basis as described above.
Then a locally integrable function $f:\mathbb{R}^d\to\mathbb{R}$ is in the Besov space $B^s_{\tau}(L_\tau(\mathbb{R}^d))$  if, and only if,
\begin{equation*}
f=\sum_{k\in\mathbb{Z}^d}\langle f,\widetilde{\phi}_k\rangle\,\phi_k
    + \sum_{i=1}^{2^{d}-1}\sum_{j\in\mathbb{N}_0}\sum_{k\in\mathbb{Z}^d}\langle f,\widetilde{\psi}_{i,j,k}\rangle\,\psi_{i,j,k}
\end{equation*}
(convergence in $\mathcal D'(\mathbb{R}^d)$) with
\begin{equation}\label{BesovNormDiscrete02}
\left(\sum_{k\in\mathbb{Z}^d}\left|\langle f,\widetilde{\phi}_k\rangle\right|^\tau\right)^{\frac{1}{\tau}}
    + \left(
    \sum_{i=1}^{2^{d}-1}\sum_{j\in\mathbb{N}_0}
    \sum_{k\in\mathbb{Z}^d}\left|\langle f,\widetilde{\psi}_{i,j,k}\rangle\right|^\tau\right)^{\frac{1}{\tau}}
    <   \infty  \, ,
\end{equation}
and \eqref{BesovNormDiscrete02}  is an equivalent (quasi-)norm  for $B^s_{\tau}(L_\tau(\mathbb{R}^d))$.
\end{prop}

\section{The stationary Stokes equation}

\subsection{The stationary Stokes equation in (weighted) Sobolev spaces}\label{Sec:StokesSystem}
In this section we collect the relevant results known so far concerning existence, uniqueness, and (weighted) Sobolev regularity of the solution to the stationary Stokes equation
\begin{equation}
\begin{alignedat}{3}
-\Delta u +  \nabla \pi & = f &&~\text{ on } \Omega, \\
  \divergenz u & = 0 &&~\text{ on } \Omega,\\
  u & =  g &&~\text{ on } \partial \Omega, 
\end{alignedat}\tag{\ref{St}} 
\end{equation}
on an arbitrary bounded Lipschitz domain $\Omega\subseteq\R^d$,  $d\geq 3$, with connected boundary.
Since we require $\divergenz u=0$, we have to make sure that the prescribed velocity field $g$ satisfies the compatibility condition
\begin{equation}\label{eq:compCondition}
\int_{\partial\Omega}g \cdot  \mathbf{n} \,\dInt\sigma =0,
\end{equation}
where  $\mathbf{n}$ is the outward unit normal to $\partial\Omega$. 
Following, for instance, \cite[Chapter IV.1]{Gal2011} we call $u\in H^1(\Omega)^d$ a \emph{(weak) solution} of \eqref{St} if $u$ is divergence free, satisfies $u=g$ on the boundary $\partial\Omega$ (in a trace sense) and fulfills the equation
\begin{equation}\label{eq:WeakFormulationSt}
 \int_\Omega \sum_{i,j=1}^d \left(\nabla u\right)_{ij}\left(\nabla \varphi\right)_{ij}(x)\,\dInt x=f(\varphi) - \int_\Omega \sum_{i=1}^d \pi(x)\frac{\partial\varphi_i(x)}{\partial x_i}\,\dInt x 
\end{equation}
for arbitrary $\varphi\in \mathcal{C}_0^\infty(\Omega)$ with a suitable pressure $\pi\in L_2(\Omega)$. 
The existence and uniqueness of such a solution $u\in H^1(\Omega)^d$ of~\eqref{St} can be guaranteed for arbitrary $f\in H^{-1}(\Omega)^d$ and $g\in H^{1/2}(\partial\Omega)^d$ satisfying \eqref{eq:compCondition}, see, e.g., \cite[Theorem IV.1.1]{Gal2011}.
The corresponding pressure $\pi\in L_2(\Omega)$ is only unique up to a constant.
In what follows, whenever we speak about ``the corresponding pressure'' to a solution $u$, we mean any of the corresponding pressures. Up to a certain degree, the solution to~\eqref{St} gets smoother, if the right hand side $f$ and the boundary value $g$ are assumed to be more regular, see \cite{AmrGir91} for instance. If we use classical Sobolev spaces to measure the smoothness, the following  proposition can be proven. Due to the linear structure of~\eqref{St}, the statement follows from \cite[Theorem~2.12]{BroShe95} together with \cite[Theorem~2.2]{BroShe95}, which relies on the results proven in~\cite{FaKeVer1988}.

\begin{prop}\label{Prop:SobRegSt}
Let $\Omega$ be a bounded Lipschitz domain in $\R^d$, $d\geq 3$, with connected boundary.  Furthermore, let $f\in L_2(\Omega)^d$ and let $g\in H^1(\partial\Omega)^d$ fulfill the condition~\eqref{eq:compCondition}.
Then there exists a unique solution $u\in H^{3/2}(\Omega)^d$ to the Stokes equation~\eqref{St} with corresponding pressure $\pi\in H^{1/2}(\Omega)$.
Moreover, the estimate 
\begin{equation}\label{eq:SobEst}
\nnrm{u}{H^{3/2}(\Omega)^d}
+
\inf_{c\in\mathbb{R}}\nnrm{\pi+c}{H^{1/2}(\Omega)}
\leq 
C\,
\grklam{\nnrm{f}{L_2(\Omega)^d}+\nnrm{g}{H^1(\partial\Omega)^d}}
\end{equation} 
holds with a constant $C\in(0,\infty)$ that depends only on $d$ and $\Omega$.
\end{prop}

If $\Omega$ is only assumed to be a bounded Lipschitz domain with connected boundary, we cannot guarantee higher regularity of the solution to~\eqref{St} in the classical Sobolev spaces, 
even if we assume the body force  $f$ and  prescribed velocity field  $g$ to be smoother than required above. This is due to boundary singularities, which can  cause the second derivatives of the solution to blow up near the boundary and  can   therefore diminish its Sobolev regularity. This effect is already known for a long time from the theory of elliptic equations on polygonal and polyhedral domains as well as on general bounded Lipschitz domains, see, e.g., \cite{Gri1985, Gri1992,JerKen1995}. However, we can capture the bad behavior of the solution at the boundary by using appropriate powers of the distance $\rho(x):=\dist(x,\partial\Omega)$ of a point $x\in\Omega$ to the boundary $\partial\Omega$. In particular, the following holds. 

 
\begin{prop}\label{Prop:WSobRegSt}
Given the setting of Proposition~\ref{Prop:SobRegSt}, the estimate  
\begin{equation}\label{eq:wSob:est}
\int_\Omega\rho(x)\cdot \abs{\nabla^2 u(x)}^2 \,\dInt x
+
\int_\Omega\rho(x)\cdot \abs{\nabla \pi(x)}^2 \,\dInt x
\leq  C
\left(\nnrm{f}{L_2(\Omega)^d}^2+\nnrm{g}{H^1(\partial\Omega)^d}^2\right)
\end{equation}
holds with a constant $C\in (0,\infty)$ that depends only on $d$ and $\Omega$. 
\end{prop}

\begin{proof}
Estimate~\eqref{eq:wSob:est} has been proven in \cite[Section 2]{BroShe95} for~\eqref{St} with zero body force. I.e., for the solution $\bar{u}\in H^1(\Omega)^d$ of the homogeneous boundary value problem
\begin{equation}\label{BVP}
\begin{alignedat}{3}
-\Delta \bar{u} +  \nabla \bar{\pi} & = 0 &&~\text{ on } \Omega, \\
  \divergenz \bar{u} & = 0 &&~\text{ on } \Omega,\\
  \bar{u} & =  g &&~\text{ on } \partial \Omega,
\end{alignedat}
\end{equation}
with corresponding pressure $\bar{\pi}\in L_2(\Omega)$, we have
\begin{equation}\label{eq:wSob:BVP}
\int_\Omega\rho(x)\cdot \abs{\nabla^2 \bar{u}(x)}^2 \,\dInt x
+
\int_\Omega\rho(x)\cdot \abs{\nabla \bar{\pi}(x)}^2 \,\dInt x
\leq  C\,
\nnrm{g}{H^1(\partial\Omega)^d}^2.
\end{equation}
In order to extend this estimate to general body forces $f\in L_2(\Omega)^d$, we argue as follows. 
Let $\widetilde{\Omega}\subseteq\R^d$ be a bounded $\cont^\infty$-domain containing the closure of $\Omega$.
Furthermore, let $\mathcal{E}\colon L_2(\Omega)\to L_2(\widetilde{\Omega})$ be a bounded extension operator, e.g., take Rychkov's extension operator from \cite{Ryc1999} combined with the restriction operator onto $\widetilde{\Omega}$.
Since the boundary of $\widetilde{\Omega}$ is smooth, the stationary Stokes equation on $\widetilde{\Omega}$ with zero Dirichlet boundary condition and body force $\mathcal{E}f\in L_2(\widetilde{\Omega})$, i.e.,  
\begin{equation*}
\begin{alignedat}{3}
-\Delta \widetilde{u} +  \nabla \widetilde\pi & = \mathcal{E}f &&~\text{ on } \widetilde\Omega, \\
  \divergenz \widetilde u & = 0 &&~\text{ on } \widetilde\Omega,\\
  \widetilde u & =  0 &&~\text{ on } \partial \widetilde\Omega, 
\end{alignedat}
\end{equation*}
has a unique solution $\widetilde{u}\in H^2(\widetilde\Omega)^d$ with pressure $\widetilde\pi\in H^1(\widetilde\Omega)$, which fulfill the estimate
\begin{equation}\label{eq:SobEst:smooth}
\nnrm{\widetilde{u}}{H^2(\widetilde\Omega)^d} 
+
\nnrm{\widetilde{\pi}}{H^1(\widetilde\Omega)} 
\leq
C\,
\nnrm{\mathcal{E}f}{L_2(\widetilde\Omega)^d}
\end{equation}
see, e.g., \cite[Theorem~3]{AmrGir91}. 
Due to the boundedness of the domain $\Omega$ and of the extension operator $\mathcal{E}$, this yields 
\begin{equation}\label{eq:wSob:ext}
\int_\Omega\rho(x)\cdot \abs{\nabla^2 \widetilde u(x)}^2 \,\dInt x
+
\int_\Omega\rho(x)\cdot \abs{\nabla \widetilde \pi(x)}^2 \,\dInt x
\leq  C\,
\nnrm{f}{L_2(\Omega)^d}^2.
\end{equation}
The linear structure of the Stokes equation~\eqref{St} allows us to split its solution $u$ into $u=\widetilde{u}|_{\Omega}+\bar{u}-u_0$ with corresponding pressure $\pi = \widetilde\pi|_{\Omega}+\bar{\pi}-\pi_0$, where 
$u_0\in H^1(\Omega)^d$ solves
\begin{equation*}
\begin{alignedat}{3}
-\Delta u_0 +  \nabla \pi_0 & = 0 &&~\text{ on } \Omega, \\
  \divergenz u_0 & = 0 &&~\text{ on } \Omega,\\
  u_0 & =  \widetilde u|_{\partial\Omega} &&~\text{ on } \partial \Omega, 
\end{alignedat}
\end{equation*}
with corresponding pressure $\pi_0\in L_2(\Omega)$. 
Such a solution $u_0$ exists by \cite[Theorem~IV.1.1]{Gal2011}, since $\widetilde u\in H^2(\widetilde\Omega)$, so that $\widetilde u|_{\partial\Omega}\in H^1(\partial\Omega)^d$ due to classical results on traces in Sobolev spaces, see, e.g., \cite{Ding96}; 
note that $\widetilde u|_{\partial\Omega}$ verifies the compatibility condition~\eqref{eq:compCondition}, since 
$$
\int_{\partial \Omega} \widetilde{u}|_{\partial \Omega}\cdot \mathbf{n} \,\dInt \sigma = \int_{\Omega} \divergenz \widetilde{u}(x) \,\dInt x =0,
$$ 
due to a proper generalization of Gauss' theorem (see \cite[Exercise II.4.3]{Gal2011}), which can be proven by \cite[Lemma II.4.1,Theorem II.3.3, Theorem II.4.1]{Gal2011}. Moreover, the pair $(u_0,\pi_0)$ verifies 
\[
\int_\Omega\rho(x)\cdot \abs{\nabla^2 u_0(x)}^2 \,\dInt x
+
\int_\Omega\rho(x)\cdot \abs{\nabla \pi_0(x)}^2 \,\dInt x
\leq  C
\nnrm{\widetilde u|_{\partial\Omega}}{H^1(\partial\Omega)^d}^2,
\]
due to the corresponding estimate from~\cite[Section~2]{BroShe95} already used above to obtain~\eqref{eq:wSob:BVP}.
Thus, since $\widetilde u$ verifies \eqref{eq:SobEst:smooth},
\[
\nnrm{\widetilde u|_{\partial\Omega}}{H^1(\partial\Omega)^d}
\leq
C \nnrm{\widetilde{u}}{H^2(\widetilde\Omega)^d}
\leq
C \nnrm{\mathcal{E}f}{L_2(\widetilde\Omega)^d}
\leq
C \nnrm{f}{L_2(\Omega)^d},
\]
so that
\[
\int_\Omega\rho(x)\cdot \abs{\nabla^2 u_0(x)}^2 \,\dInt x
+
\int_\Omega\rho(x)\cdot \abs{\nabla \pi_0(x)}^2 \,\dInt x
\leq  C
\nnrm{f}{L_2(\Omega)^d}^2.
\]
Since all the constants used in this proof depend only on $d$ and $\Omega$, the last estimate, together with~\eqref{eq:wSob:BVP} and \eqref{eq:wSob:ext} proves the assertion.
\end{proof} 

\subsection{Besov regularity for the stationary Stokes equation}\label{Sec:stokes}

In this section we prove the following main result concerning the Besov regularity in the scale~\eqref{eq:scale} of the solution to the Stokes equation~\eqref{St} and of the corresponding pressure.

\begin{theo}\label{Th:BesovStokes}
Let $\Omega$ be a bounded Lipschitz domain in $\R^d$, $d\geq 3$, with connected boundary.  Let $u $ be the unique solution of \eqref{St} with $f\in L_2(\Omega)^d$ and $g\in H^1(\partial\Omega)^d$ fulfilling additionally~\eqref{eq:compCondition}, and let $\pi$ be the corresponding pressure.  Then
\begin{equation}\label{eq:BesovRegu}
u\in B^{s_1}_{\tau_1}(L_{\tau_1}(\Omega))^d, 
\quad \frac{1}{\tau_1}=\frac{s_1}{d}+\frac{1}{2},\quad 
0< s_1<\min\ssggklam{\frac{3}{2}\cdot\frac{d}{d-1},2},
\end{equation}
and 
\begin{equation}\label{eq:BesovRegp}
\pi\in B^{s_2}_{\tau_2}(L_{\tau_2}(\Omega)),
\quad \frac{1}{\tau_2}=\frac{s_2}{d}+\frac{1}{2},\quad 
0< s_2< \frac{1}{2}\cdot\frac{d}{d-1}.
\end{equation}
Moreover, for this range of parameters, the estimate
\begin{equation}\label{eq:NormEstBesovStokes}
\nnrm{u}{B^{s_1}_{\tau_1}(L_{\tau_1}(\Omega))^d}
+
\inf_{c\in\mathbb{R}} \nnrm{\pi+c}{ B^{s_2}_{\tau_2}(L_{\tau_2}(\Omega))} 
\leq C\,
\grklam{\nnrm{f}{L_2(\Omega)^d}+ \nnrm{g}{H^1(\partial\Omega)^d}}
\end{equation}
holds with a constant $C\in (0,\infty)$ that depends only on $\Omega$,  $d$, $s_1$, $s_2$, $\tau_1$, and $\tau_2$.
\end{theo}

Before we prove this statement, we want to emphasize its significance for the question raised in the introduction, whether adaptivity pays or not for the numerical treatment of the Stokes equation. 
Moreover, we relate our result to what is already known about the Besov regularity of the Stokes equation in the scale~\eqref{eq:scale}.

\begin{bem}\label{rem:MitreaAndWe}
\begin{enumerate}[align=right,label=\textup{(\roman*)}]
\item 

If we only assume that the boundary of the underlying domain $\Omega$ is Lipschitz (and connected), then to the best of our knowledge the Sobolev regularity result presented in Proposition~\ref{Prop:SobRegSt} is sharp, i.e. for arbitrary $\varepsilon>0$, there exists a bounded Lipschitz domain $\Omega_\varepsilon$ and a function $f\in L_2(\Omega_\varepsilon)^d$, such that the solution $u$ to \eqref{St} with $g=0$ fails to have $L_2$-Sobolev regularity of order $3/2+\varepsilon$. However, Theorem~\ref{Th:BesovStokes} shows that for arbitrary Lipschitz domains we can go beyond $3/2$ in the scale~\eqref{eq:scale} of Besov spaces. 
For $d=3$ we can choose any $s_1$ less than $2$,  whereas for $d\geq 4$ the bound is given by $3/2\cdot d/(d-1)$, which is strictly greater than $3/2$.
As already mentioned in the introduction, this justifies the usage of adaptive numerical methods for the Stokes equation in the sense that in this situation they can  have a higher convergence rate than their classical uniform alternatives.
The same is true for the pressure $\pi$, since its Besov regularity in the scale~\eqref{eq:scale} is $d/(d-1)$ times higher than its worst case Sobolev regularity.

\item 
To the best of our knowledge, the most far reaching results concerning the Besov regularity of the solution $u$ to~\eqref{St} on general bounded Lipschitz domains (and for the corresponding pressure $\pi$), have been obtained in~\cite{MitMat2012}.
Using boundary integral methods, the authors undertake a detailed analysis of the Besov regularity of the boundary value problem for the Stokes equation.
Among others, for arbitrary dimensions $d\geq 2$, they determine corresponding ranges of smoothness and integrability parameters allowing for implications of the type
\[
f\in B^{s-2}_q(L_p(\Omega))^d,
g\in B^{s-1/p}_q(L_p(\partial\Omega))^d
\quad\Longrightarrow\quad
u\in B^s_q(L_p(\Omega))^d,
\pi\in B^{s-1}_q(L_p(\Omega))
\]
for Lipschitz domains $\Omega\subseteq\R^d$, see~\cite[Theorem~10.15]{MitMat2012}. However, these ranges of admissible parameters depend on the degree of roughness of the boundary $\partial\Omega$, which is described by a value $\varepsilon=\varepsilon(\Omega)\in (0,1]$: the smaller the $\varepsilon$, the rougher the underlying domain and the smaller the admissible range $\mathcal{R}_{d,\varepsilon}$.
Theorem~\ref{Th:BesovStokes} supports the claim from~\cite{MitMat2012} that the results therein are sharp for low dimensions $d=2,3$.
However, for higher dimensions $d\geq 4$, on general bounded Lipschitz domains  with connected boundary, i.e., if we do not make any further assumptions on the smoothness of the boundary, we obtain higher regularity in the scale~\eqref{eq:scale} than what is possible to extract from~\cite{MitMat2012}.
In detail, we have the following relationship between the two results.


\begin{itemize}
\item If $d\geq 4$ and $\varepsilon\leq 1/2\cdot 1/(d-1)$, i.e., if the underlying domain $\Omega$ is rough, the results in~\cite{MitMat2012} do not  imply a \cef Besov regularity higher than $3/2$ for the solution and $1/2$ for the corresponding pressure. Even choosing the integrability parameter less than $2$, does not help.
However, we can get higher with our result and reach any Besov regularity $s_1<3/2\cdot d/(d-1)$ in the scale~\eqref{eq:scale} for the solution and $s_2<1/2\cdot d/(d-1)$ for the pressure. 
Our findings are depicted in Figure~\ref{fig:DeVoreTriebelD4} by means of a so called DeVore-Triebel diagram. A point $(1/\tau,s)$ in the first quadrant stands for the Besov space $B^s_\tau(L_\tau(\Omega))^d$.
In particular, the ray with slope $d$ starting in $(1/2,0)$ represents the scale~\eqref{eq:scale} of Besov spaces.
Due to Theorem~\ref{Th:BesovStokes}, the regularity of the solution to~\eqref{St} climbs up this scale until it (almost) reaches the smoothness $s^*=3/2\cdot d/(d-1)$.  Firstly, we observe, as already discussed in the first part of this remark, that the Besov regularity in the scale \eqref{eq:scale} is higher than the Sobolev regularity. Secondly this Besov regularity is more than we can extract from~\cite{MitMat2012}.
If we do not impose any further assumption on the  Lipschitz character of the domain,  the results in~\cite{MitMat2012} merely guarantee that the solution is contained in those Besov spaces that correspond to the dark shaded area (and to any point below and to its right until reaching the line $\{(1/\tau,s)\colon 1/\tau=s/d+1\}$ due to standard embeddings), see, e.g.,~\cite[Theorem~2.61]{Cio14}. The only possibility to enlarge the admissible range of parameters by using~\cite{MitMat2012} is to impose more regularity on the boundary $\partial\Omega$ of the domain, i.e., to relax the conditions on $\varepsilon=\varepsilon(\Omega)$.
In Figure~\ref{fig:DeVoreTriebelD4}, this adds the light shaded area to the range covered by~\cite{MitMat2012}, see, in particular, Theorem~10.15 therein.
However, if $\varepsilon\leq 1/2\cdot 1/(d-1)$, this area does not include all Besov spaces from the scale~\eqref{eq:scale} with smoothness parameter less than $s^*$.

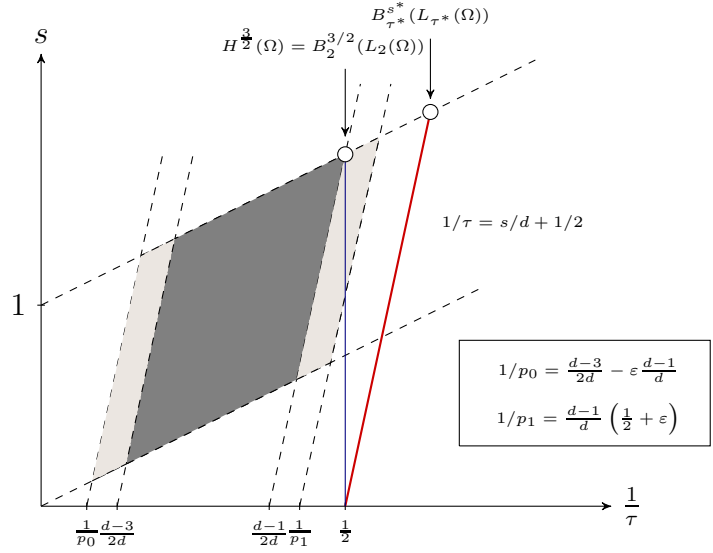
\begin{figure}
\begin{center}
\begin{tikzpicture}
\draw [->,>=stealth'] (0,0) -- (0,6) node[above] {$s$};
\draw [->,>=stealth'] (0,0) -- (7.5,0) node[right] {$\frac{1}{\tau}$};

\draw [dashed](3,0) -- (4.2,5.59992); 
\draw [dashed](1,0) -- (2,4.6666); 

\draw [dashed](0.6,0) -- (1.6,4.6666); 
\draw [dashed](3.4,0) -- (4.6,5.59992); 

\draw[dashed] (0,0) -- (5.75,2.875); 
\draw[dashed] (0,2.666) -- (6.5,5.91666); 


\draw [dashed, fill=lightgray](1.12,0.56) -- (3.36,1.68) -- (4,4.668) -- (1.76,3.548);

\draw [dashed,fill=verylightgray](3.36,1.68) -- (3.808,1.904) -- (4.448,4.892) -- (4,4.668);


\draw [dashed, fill=verylightgray](0.672,0.336) -- (1.12,0.56) -- (1.76,3.548) -- (1.312,3.324);


\draw [red,thick] (4,0) -- (5.12,5.228); 
\draw (6.2,3.5) node[fill=white, anchor=south] (BesL) {\tiny{$1/\tau=s/d+1/2$}};

\draw (1,0.05) -- (1,-0.05) node[anchor=north] {\tiny $\frac{d-3}{2d}$};
\draw (3,0.05) -- (3,-0.05) node[anchor=north] {\tiny $\frac{d-1}{2d}$};

\draw (4,0.05) -- (4,-0.05) node[anchor=north] {\tiny $\frac{1}{2}$};

\draw (5.12,5.228) node[draw=black, circle,fill=white,inner sep=2pt] (BesPoint) {};

\draw [<-,>=stealth'] (5.12,5.4) -- (5.12,6.2)  node[above] {\tiny $B^{s^*}_{\tau^*}(L_{\tau^*}(\Omega))$
};

\draw [blue](4,0) -- (4,4.6666);
\draw (4,4.6666) node[draw=black,circle,fill=white,inner sep=2pt] (SobPoint) {};

\draw [<-,>=stealth'] (4,4.9) -- (4,5.8)  node[above] { \tiny $H^{\frac{3}{2}}(\Omega)=B^{3/2}_2(L_2(\Omega))\quad\quad\:$ };

\draw (0.6,0.05) -- (0.6,-0.05) node[anchor=north] {\tiny $\frac{1}{p_0}$};
\draw (7.2,1.5) node[fill=white, anchor=south] (NstE) {\tiny $1/{p_0}=\frac{d-3}{2d}-\varepsilon\frac{d-1}{d}$};
\draw (3.4,0.05) -- (3.4,-0.05) node[anchor=north] {\tiny $\frac{1}{p_1}$};
\draw (7.2,0.8) node[fill=white, anchor=south] (NstE) {\tiny $1/{p_1}=\frac{d-1}{d}\left(\frac{1}{2}+\varepsilon\right)$};

\draw (5.5,0.8) -- (8.8,0.8) -- (8.8,2.2) -- (5.5,2.2) -- (5.5,0.8);

\draw (0.05,2.666) -- (-0.05,2.666) node[anchor=east] {$1$};

\end{tikzpicture}
\caption[DeVoreTriebel002]{\begin{tabular}[t]{l}Besov regularity for  the \icp solution $u$ to~\eqref{St} achieved  by  exploiting \cite{MitMat2012}, \\   versus the results in Theorem \ref{Th:BesovStokes}, illustrated in a DeVore/Triebel dia-\\ 
gram ($d\geq 4$). \end{tabular}}\label{fig:DeVoreTriebelD4}
\end{center}
\end{figure}

\item If $d\geq 4$ and $\varepsilon> 1/2\cdot 1/(d-1)$, i.e., if the boundary $\partial\Omega$ is smooth enough, then, the same regularity as in Theorem~\ref{Th:BesovStokes} can be established by exploiting~\cite[Theorem~10.15]{MitMat2012}; with slightly weaker assumptions on $f$ and $g$.

\item If $d=3$, Theorem~\ref{Th:BesovStokes} is covered by~\cite[Theorem~10.15]{MitMat2012}, which guarantees that, even under weaker assumptions on the data, the solution $u\in B^{2-\delta}_{p}(L_{p}(\Omega))$ for arbitrary small $\delta>0$ and suitable $p=p(\delta)>1$. 
Consequently, $u$ has Besov regularity of any order  $s_1<2$ in the scale~\eqref{eq:scale} due to standard embeddings of Besov spaces.
If $\varepsilon >1/2$, then under slightly stronger assumptions on the data, any regularity up to $9/4$ is possible due to~\cite[Theorem~10.15]{MitMat2012}.
The pressure has the corresponding regularity $s_2<1$ and $s_2<5/4$, respectively.
\end{itemize}

\end{enumerate}
\end{bem}

As already pointed out, the Sobolev regularity of the solution to the Stokes equation is limited by $3/2$ in a worst case scenario. 
However, we know from Proposition~\ref{Prop:WSobRegSt} that we can still guarantee integrability of the second derivatives if multiplied by a proper power of the distance to the boundary. These relationships will be the most important ingredients for the proof of the following embedding, which, together with Proposition~\ref{Prop:WSobRegSt}, proves Theorem~\ref{Th:BesovStokes} as we will explain at the end of the Section.
Recall that
\[
W^m_\alpha(L_p(\Omega))
=
\left\{f\in L_p(\Omega): 
 \nnrm{f}{W^m_\alpha(L_p(\Omega))}^p
=
\nnrm{f}{L_p(\Omega)}^p+\int_\Omega \rho(x)^\alpha \abs{\nabla^m f(x)}_{\ell_p}^p \dInt x<\infty \right\}, 
\]
for $m\in\bN_0$, $\alpha>0$ and $p\in(1,\infty)$.
  
\begin{theo}\label{Th:imbedding}
Let $\Omega$ be a bounded Lipschitz domain in $\R^d$, $d\geq 3$. Let $\alpha_0>0$,  $\alpha >0,$ and $\gamma\in\mathbb{N}$ with $\alpha <2\gamma$.
Then,
\[
H^{\alpha_0}(\Omega)\cap W^\gamma_\alpha(L_2(\Omega))
\hookrightarrow
B^s_{\tau}(L_\tau(\Omega)),
\qquad
\frac{1}{\tau}=\frac{s}{d}+\frac{1}{2},
\]
for all 
\[
0<s<\min\ssggklam{\frac{2\gamma -\alpha}{2}\cdot\frac{d}{d-1},\alpha_0\cdot\frac{d}{d-1},\gamma}.
\]
\end{theo}

\begin{proof}  The proof can essentially be performed by following the line of \cite{DahDeV1997}. For reader's convenience, we briefly discuss the basic steps. Let us fix $s$ and $\tau$ as stated in the theorem. We choose a suitable wavelet basis
$$
\ggklam{
\phi_k,\psi_{i,j,k}: (i,j,k)\in \{1,...,2^d-1\}\times\mathbb{N}_0\times \mathbb{Z}^d
} 
$$
of $L_2(\mathbb{R}^d)$ satisfying the assumptions from Section~\ref{Sec:BesovWavelet}  with $r>\gamma$. 
This means, in particular, that there exists a cube $Q$ centered at the origin such that for $(j,k)\in \mathbb{N}_0 \times \mathbb{Z}^d$ the cube
$$Q_{j,k}:=2^{-j}k+ 2^{-j}Q $$
contains the support of $\psi_{i,j,k}$ for all $i\in \{1,...,2^d-1\}$ and $\supp \phi_k\subseteq Q_{0,k}$; remember that the supports of the corresponding dual basis fulfill the same requirements. 
Fix $v\in H^{\alpha_0}(\Omega)=B^{\alpha_0}_{2}(L_2(\Omega))$. Since $\Omega$ is assumed to have a Lipschitz continuous boundary, there exists a  linear bounded extension operator $\mathcal{E}\colon B^{\alpha_0}_2(L_2(\Omega))\to B^{\alpha_0}_2(L_2(\mathbb{R}^d))$,
see, e.g., \cite{Ryc1999}.
Due to Proposition~\ref{Prop:CharacterizationBesov1}, we have the following wavelet expansion 
\[
\mathcal{E}v
=
\sum_{k\in\mathbb{Z}^d}\langle \mathcal{E}v,\widetilde{\phi}_k\rangle\,\phi_k
+ 
\sum_{i=1}^{2^{d}-1}\sum_{j\in\mathbb{N}_0}\sum_{k\in\mathbb{Z}^d}\langle \mathcal{E}v,\widetilde{\psi}_{i,j,k}\rangle\,\psi_{i,j,k}
\]
for the extended $v$. If we restrict to the scaling functions and wavelets associated to those cubes $Q_{j,k}$ that have a non-empty intersection with $\Omega$, i.e., if we consider only the indexes from
$$\Lambda:=\left\{ (i,j,k)\in \{1,...,2^d-1\}\times\mathbb{N}_0 \times \mathbb{Z}^d : Q_{j,k} \cap\Omega\neq \emptyset\right\} $$
and
$$\Gamma:=\{k\in\mathbb{Z}^d: Q_{0,k}\cap \Omega \neq \emptyset \}, $$
then
\[
v= \sum_{k\in\Gamma}\langle \mathcal{E}v,\widetilde{\phi}_k\rangle\,\phi_k
+ 
\sum_{(i,j,k)\in \Lambda}\langle \mathcal{E}v,\widetilde{\psi}_{i,j,k}\rangle\,\psi_{i,j,k} 
\]
still holds on $\Omega$. 
Consequently, due to Proposition~\ref{Prop:CharacterizationBesov2} and since Besov spaces on domains are defined via restriction, in order to prove the theorem, it is enough to show that
\begin{equation}\label{eq:est:scalingfcts}
\ssgrklam{\sum_{k\in \Gamma}\abs{\langle \mathcal E v,\widetilde{\phi}_k\rangle}^\tau}^{\frac{1}{\tau}}
\leq C
\nnrm{v}{H^{\alpha_0}(\Omega)}
\end{equation}
and
\begin{equation}\label{eq:est:wavelets}
\ssgrklam{
\sum_{(i,j,k)\in\Lambda}\abs{\langle \mathcal{E} v,\widetilde{\psi}_{i,j,k}\rangle}^\tau}^{\frac{1}{\tau}}
\leq C
\grklam{
\nnrm{v}{H^{\alpha_0}(\Omega)} + \nnrm{v}{W^\gamma_\alpha(L_2(\Omega))}
},
\end{equation}
with appropriate constants $C\in (0,\infty)$ that do not depend on $v$. For simplicity, we make slight abuse of notation and write $v$ instead of $\mathcal{E}v$ in the sequel.

We start with~\eqref{eq:est:scalingfcts}. The index set $\Gamma$ is finite due to the boundedness of the underlying domain, so that a simple application of Jensen's inequality, followed by an application of Proposition~\ref{Prop:CharacterizationBesov1} together with the boundedness of the extension operator and the equivalence of the norms $\nnrm{\cdot}{H^{\alpha_0}(\Omega)}$ and $\nnrm{\cdot}{B^{\alpha_0}_{2}(L_2(\Omega))}$ on $B^{\alpha_0}_{2}(L_2(\Omega))$, yields
\[
\sum_{k\in \Gamma}\abs{\langle v,\widetilde{\phi}_k\rangle}^\tau
\leq
C 
\ssgrklam{\sum_{k\in \Gamma}\abs{\langle v,\widetilde{\phi}_k\rangle}^2}^{\frac{\tau}{2}}
\leq
C
\nnrm{v}{B^{\alpha_0}_2(L_2(\R^d))}^\tau
\leq
C
\nnrm{v}{B^{\alpha_0}_2(L_2(\Omega))}^\tau
\leq
C
\nnrm{v}{H^{\alpha_0}(\Omega)}^\tau.
\]
The constants above as well as all the other constants appearing in this proof do not depend on $v$.

To prove the second estimate~\eqref{eq:est:wavelets}, we split the sum on the left hand side into two parts and consider those coefficients that are related to wavelets with support in the interior of $\Omega$ isolated from those associated with wavelets that might have support on the boundary $\partial\Omega$ or outside of $\Omega$.
By using the notations
\begin{align*}
\rho_{j,k} & := \dist(Q_{j,k},\partial \Omega), \\
\Lambda_j & :=\{(i,l,k)\in \Lambda: l=j\}, \\
\Lambda_{j,m} & :=\{(i,j,k)\in\Lambda_j: m\cdot 2^{-j}\leq \rho_{j,k} <(m+1)\cdot 2^{-j}\},\\
\Lambda_j^0 & := \Lambda_j \backslash \Lambda_{j,0}, \\
\Lambda^0 & :=\bigcup_{j\in\mathbb{N}_0}\Lambda_j^0,
\end{align*}
for $j,m\in\mathbb{N}_0$, $k\in\mathbb{Z}^d$, this splitting can be written as
\begin{equation}\label{eq:est:split}
\sum_{(i,j,k)\in\Lambda}
\abs{\langle v,\widetilde\psi_{i,j,k}\rangle}^\tau 
=
\sum_{(i,j,k)\in\Lambda^0}
\abs{\langle v,\widetilde \psi_{i,j,k}\rangle}^\tau
+
\sum_{(i,j,k)\in\Lambda\backslash \Lambda^0}
\abs{\langle v,\widetilde \psi_{i,j,k}\rangle}^\tau
=:I+I\!\!I.
\end{equation}
\icp

To estimate $I$, we exploit a Whitney type estimate (see \cite[Theorem 3.4]{DeVSha1984}) to obtain 
\begin{align*}
\abs{ \langle v,\widetilde{\psi}_{i,j,k}\rangle}
& \leq C\, 2^{-j\gamma}\rho_{j,k}^{-\alpha/2}\left(\int_{Q_{j,k}}\abs{\rho(x)^{\alpha/2}\nabla^\gamma v(x)}_{\ell_2}^2\,\dInt x\right)^{1/2}\\
& =: C\, 2^{-j\gamma}\rho_{j,k}^{-\alpha/2}\,\mu_{j,k},
\end{align*}
with a finite constant $C$ that does not depend on $j$ or $k$. Fix $j\in\mathbb{N}_0$. Exploiting H\"older's inequality with parameters $2/\tau >1$ and $2/(2-\tau)$, we obtain
\begin{equation}\label{eq:estimateWaveletCoeff}
\sum_{(i,j,k)\in \Lambda_j^0} \abs{\langle v,\widetilde \psi_{i,j,k}\rangle}^\tau 
\leq C \,\ssgrklam{\sum_{(i,j,k)\in \Lambda_j^0}\mu_{j,k}^2}^{\frac{\tau}{2}} \ssgrklam{\sum_{(i,j,k)\in \Lambda_j^0}2^{- \frac{2j\tau \gamma}{2-\tau}}\rho_{j,k}^{-\frac{\alpha \tau}{2-\tau}}}^{\frac{2-\tau}{2}}.
\end{equation}
The first factor on the right hand side can be estimated by a constant times $\nnrm{v}{W^\gamma_\alpha(L_2(\Omega))}^{\tau}$, which is bounded by assumption (see, e.g., the proof of \cite[Theorem~4.7]{Cio14} for details).
In order to estimate the second factor we use the Lipschitz character of $\Omega$, which guarantees that
\begin{equation}\label{EstimateIndexSet1}
|\Lambda_{j,m}|
\leq C\, 2^{j(d-1)}\quad \text{for all}\quad j,m\in\mathbb{N}_0.
\end{equation}
Moreover, the boundedness of $\Omega$ yields
 $\Lambda_{j,m}=\emptyset$  for all $j,m\in\mathbb{N}_0$ with $m\geq C\cdot 2^j$. 
The constant $C$ in both estimates does not depend on $j$ or $m$. 
Consequently, we obtain
\begin{equation}\label{eq:est:sharply}
\rrklam{\sum_{(i,j,k)\in \Lambda_j^0}2^{- \frac{2j\tau \gamma}{2-\tau}}\rho_{j,k}^{-\frac{\alpha \tau}{2-\tau}}}^{\frac{2-\tau}{2}} 
\leq C\,
\sgrklam{2^{j\left( d-1-\frac{(2\gamma-\alpha)\tau}{2-\tau}\right)} + 2^{j\left( d-\frac{2\gamma\tau}{2-\tau}\right)} }^{\frac{2-\tau}{2}}.
\end{equation}
If $\alpha\tau/(2-\tau)>1$, the last estimate follows from the convergence of the harmonic series. 
For $0<\alpha\tau/(2-\tau)\leq 1$, it can be obtained by estimating the integral $\int_1^{C 2^j}t^{-\frac{\alpha\tau}{2-\tau}}\,\dInt t$ properly. 
Summing up over $j\in\bN_0$ in~\eqref{eq:estimateWaveletCoeff} and using~\eqref{eq:est:sharply}, leads to
\[
\sum_{(i,j,k)\in \Lambda^0} \abs{\langle v,\widetilde \psi_{i,j,k}\rangle}^\tau 
\leq C
\nnrm{v}{W^\gamma_\alpha(L_2(\Omega))}^\tau
\sum_{j\in\bN_0}\sgrklam{2^{j\left( d-1-\frac{(2\gamma-\alpha)\tau}{2-\tau}\right)} + 2^{j\left( d-\frac{2\gamma\tau}{2-\tau}\right)} }^{\frac{2-\tau}{2}}.
\]
Obviously, the sums on the right hand side converge if, and only if, $0<s<\min\sggklam{\gamma,\frac{2\gamma-\alpha}{2}\frac{d}{d-1}}$.
Therefore, since this is part of our assumptions,
\begin{equation}\label{eq:est:I}
 \sum_{(i,j,k)\in \Lambda^0}\abs{\langle v,\widetilde \psi_{i,j,k}\rangle}^\tau 
\leq C\,
\nnrm{v}{W^\gamma_\alpha(L_2(\Omega))}^\tau.
\end{equation}

In the last step we have to estimate the second term $I\!\!I$ from~\eqref{eq:est:split}. 
To this end we use H\"older's inequality together with~\eqref{EstimateIndexSet1} to verify that for every $j\in\bN_0$,
$$
\sum_{(i,j,k)\in\Lambda_{j,0}}\abs{\langle v,\widetilde\psi_{i,j,k}\rangle}^\tau 
\leq 
C \,2^{j(d-1)\frac{2-\tau}{2}} 
\ssgrklam{\sum_{(i,j,k)\in\Lambda_{j,0}}\abs{\langle v,\widetilde\psi_{i,j,k}\rangle}^2}^{\frac{\tau}{2}}, $$
where $C\in(0,\infty)$ is a constant independent of $j$. Summing up over all $j\in\mathbb{N}_0$ and using the relationship $1/\tau=s/d+1/2$ yields 
\begin{equation*}
\sum_{(i,j,k)\in\Lambda\backslash \Lambda^0}
\abs{\langle v,\widetilde\psi_{i,j,k}\rangle}^\tau  
 \leq C\sum_{j\in \mathbb{N}_0} \ssgrklam{2^{j\frac{d-1}{d}\cdot s\cdot \tau}\ssgrklam{\sum_{(i,j,k)\in\Lambda_{j,0}}\abs{\langle v,\widetilde\psi_{i,j,k}\rangle}^2}^{\frac{\tau}{2}}}\\
 \leq C\, \nnrm{v}{B^{\frac{d-1}{d}s}_\tau(L_2(\mathbb{R}^d))}^\tau,
\end{equation*}
where the last estimate is due to Proposition~\ref{Prop:CharacterizationBesov1}.
Since $B^{\alpha_0}_2(L_2(\R^d)) \hookrightarrow B^{\frac{d-1}{d}s}_\tau(L_2(\R^d)) $ for arbitrary $0<s<\alpha_0\frac{d}{d-1}$, see, e.g., \cite[Proposition~2.3.2.2]{Tri1983}, we obtain
\begin{align*}
\sum_{(i,j,k)\in\Lambda\backslash \Lambda^0}
\abs{\langle v,\widetilde\psi_{i,j,k}\rangle}^\tau  
\leq C\, \nnrm{v}{B^{\alpha_0}_2(L_2(\mathbb{R}^d))}^\tau
\leq C\, \nnrm{v}{B^{\alpha_0}_2(L_2(\Omega))}^\tau
\leq C\, \nnrm{v}{H^{\alpha_0}(\Omega)}^\tau.
\end{align*}
This estimate together with~\eqref{eq:est:I} prove that~\eqref{eq:est:wavelets} holds and, therefore, so does our assertion.
\end{proof}

\begin{proof}[Proof of Theorem~\ref{Th:BesovStokes}]
 Due to Proposition~\ref{Prop:SobRegSt} and Proposition~\ref{Prop:WSobRegSt}, the statement follows from a straightforward application of the embedding obtained above in Theorem~\ref{Th:imbedding}.
\end{proof}

\section{Besov regularity for the stationary Navier-Stokes equation}\label{Sec:Navierstokes}

In this section we extend our analysis to the stationary Navier-Stokes equation
\begin{equation}
\begin{alignedat}{3}
-\Delta u +  \nu u\cdot(\nabla u)+ \nabla \pi & = f &&~\text{ on } \Omega, \\
  \divergenz u & = 0 &&~\text{ on } \Omega,\\
  u & =  g &&~\text{ on } \partial \Omega; 
\end{alignedat}\tag{\ref{eq:NavierStokes}} 
\end{equation}
%
$\nu >0$ denotes the Reynolds number and, as before, $g$ is assumed to fulfil \eqref{eq:compCondition} for compatibility reasons. Following the lines of the previous sections, we understand \eqref{eq:NavierStokes} in a weak sense and call $u\in H^1(\Omega)^d$ a \emph{(weak) solution} of \eqref{eq:NavierStokes} if $u$ is divergence free, satisfies $u=g$ on the boundary $\partial\Omega$ (in a trace sense) and fulfils the equation
\[
\int_\Omega \sum_{i,j=1}^d \left(\nabla u\right)_{ij}\left(\nabla \varphi\right)_{ij}\, \dInt x 
+ 
\nu\int_\Omega (u\cdot (\nabla u))\,\varphi\,\dInt x 
= 
 f(\varphi) 
-
\int_\Omega \sum_{i=1}^d \pi\,\frac{\partial\varphi_i}{\partial x_i}\,\dInt x
\] 
for all $\varphi\in\mathcal{C}_0^\infty(\Omega)$ with a suitable $\pi\in L_2(\Omega)$. See \cite[Chapter IX]{Gal2011} for more details. 

Our goal is to establish existence of a solution to~\eqref{eq:NavierStokes} with high regularity in the scale~\eqref{eq:scale} of Besov spaces.
To this end we exploit what we proved in the previous section about the regularity of the Stokes equation together with a fixed point argument.
Our approach is based on the following basic observation:
Assume that $u\in H^1(\Omega)^d$ is a solution to~\eqref{eq:NavierStokes} with $f\in L_2(\Omega)^d$ and $g\in H^1(\partial\Omega)^d$. Then, if we could guarantee that $u\cdot(\nabla u)\in L_2(\Omega)^d$, the solution $u\in H^1(\Omega)^d$ to~\eqref{eq:NavierStokes} would actually be a solution to~\eqref{St} with body force $f-u\cdot(\nabla u)\in L_2(\Omega)^d$ (instead of $f$) and prescribed velocity field $g\in H^1(\partial\Omega)^d$.
As a consequence, $u$ and the corresponding pressure $\pi$ would have the Besov regularity in the scale~\eqref{eq:scale} guaranteed by Theorem~\ref{Th:BesovStokes}.
Of course, $u\in H^1(\Omega)^d$ is not a sufficient condition for $u\cdot(\nabla u)\in L_2(\Omega)^d$. 
However, the latter would hold if we could additionally guarantee that our solution is essentially bounded. 
This would definitively be fulfilled if we require $u\in B^t_p(L_p(\Omega))^d$ for some $p>1$ and $t\notin\bN$ with $t>d/p$, since in this case $B^t_p(L_p(\Omega))=W^t(L_p(\Omega)) \hookrightarrow L_\infty(\Omega)$ due to Sobolev's embedding theorem.
A solution to~\eqref{St} with this property can be obtained by exploiting the results from~\cite{MitMat2012}.
Theorem~10.15 therein guarantees, among others, that the Stokes problem~\eqref{St} with body force $f\in B^{t-2}_p(L_p(\Omega))$ and boundary condition $g\in B^{t-1/p}_p(L_p(\partial\Omega))^d$ fulfilling~\eqref{eq:compCondition} has a unique solution $u\in B^t_p(L_p(\Omega))^d$ with corresponding pressure $\pi\in B^{t-1}_p(L_p(\Omega))$---provided the parameters $p$ and $t$ are within an admissible range $\mathcal{R}_{d,\varepsilon}$ that depends on the dimension $d$ and on the Lipschitz character of the underlying domain $\Omega$, which is expressed by the quantity $\varepsilon=\varepsilon(\Omega)\in (0,1]$, see also Remark \ref{rem:MitreaAndWe}. Moreover, there exists a finite constant $C>0$, such that
\begin{align*}
\nnrm{u}{B^t_p(L_p(\Omega))^d}
+
\inf_{c\in\R}\nnrm{\pi+c}{B^{t-1}_p(L_p(\Omega))}
\leq C
\grklam{
\nnrm{f}{B^{t-2}_p(L_p(\Omega))^d}
+
\nnrm{g}{B^{t-1/p}_p(L_p(\partial\Omega))^d}
}.
\end{align*}
For $d=3$, there always exists a pair $(p,t)$ with $t>d/p$ and $t\notin\bN$ within the range $\mathcal{R}_{3,\varepsilon}$ of parameters admissible in~\cite[Theorem~10.15]{MitMat2012}, since $\mathcal{R}_{3,\varepsilon}$ covers all $p>2$ and $t\in\R$ such that
\begin{equation}\label{eq:range:cond:3}
\max\ssggklam{\frac{3}{p},1} < t< \min\ssggklam{\frac{3}{p}+ \varepsilon ,\,\, 1+\frac{1}{p}}.
\end{equation} 
This is also the case for $d\geq 4$ if the underlying domain is smooth enough, i.e., if we assume that the quantity $\varepsilon=\varepsilon(\Omega)$ describing the Lipschitz character of $\Omega$ fulfils
\begin{equation*}
\varepsilon > \frac{d-3}{2(d-1)}.
\end{equation*} 
Then the corresponding range $\mathcal{R}_{d,\varepsilon}$ in~\cite[Theorem~10.15]{MitMat2012} covers all $p>d-1$ and $t\in\R$ such that
\begin{equation}\label{eq:range:cond:d}
\max\ssggklam{\frac{d}{p},1} < t< \min\ssggklam{\frac{d}{p}+ (d-1)\varepsilon -\frac{d-3}{2},\,\, 1+\frac{1}{p}}.
\end{equation}
Thus, for these ranges of parameters \cite[Theorem~10.15]{MitMat2012} guarantees that the linear solution operator of the Stokes problem~\eqref{St},
\begin{align*}
L:=L_{t,p,\Omega}\colon B^{t-2}_{p}(L_p(\Omega))^d\times B^{t-1/p,0}_{p}(L_p(\partial\Omega))^d &\to B^t_p(L_p(\Omega))^d\times \rklam{B^{t-1}_p(L_p(\Omega))/\R_\Omega}\\
(f,g)&\mapsto L(f,g):=(u,\pi),
\end{align*}
where $u$ is the unique solution to~\eqref{St} with body force $f$ and boundary value $g$, and $\pi$ is the corresponding pressure, is well-defined and bounded (see Section~\ref{Sec:Notation} for notation). We denote its operator norm by
\[
\nrm{L_{t,p,\Omega}}:= \sup_{(f,g)\in Y\setminus \{0\}} \frac{\nnrm{L_{t,p,\Omega}(f,g)}{X_{t,p,\Omega}}}{\nnrm{(f,g)}{Y_{t,p,\Omega}}},
\]
where
\[
X:=X_{t,p,\Omega}:= B^t_p(L_p(\Omega))^d\times \rklam{B^{t-1}_p(L_p(\Omega))/\R_\Omega}
\]
and 
\[
Y:=Y_{t,p,\Omega}:=B^{t-2}_{p}(L_p(\Omega))^d\times  B^{t-1/p,0}_{p}(L_p(\partial\Omega))^d.
\]
Using this notation, we can present our main result concerning the regularity of the solution to~\eqref{eq:NavierStokes} in the scale~\eqref{eq:scale} of Besov spaces.

\begin{theo}\label{Th:NavierStokesDG2}
Let $\Omega$ be a bounded Lipschitz domain in $\R^d$, $d\geq 3$, with connected boundary. Assume that the quantity $\varepsilon=\varepsilon(\Omega)\in(0,1]$ from~\cite{MitMat2012} describing the Lipschitz character of $\Omega$ fulfils
\begin{equation}\label{eq:eps:Lip}
\varepsilon > \frac{d-3}{2(d-1)}.
\end{equation}
Let $p> d-1$ and $t\in\R$ satisfy~\eqref{eq:range:cond:3} for $d=3$ and \eqref{eq:range:cond:d} for $d\geq 4$, respectively.
Fix 
\[
f\in L_{2}(\Omega)^d\cap B_p^{t-2}(L_p(\Omega))^d 
\]
and
\[
g\in H^1(\partial\Omega)^d\cap  B_p^{t-1/p}(L_{p}(\partial\Omega))^d\quad\text{with}\quad\int_{\partial\Omega}g \cdot  \mathbf{n} \,\dInt\sigma =0.
\]
Then there exists a finite constant $C=C_{t,p,\Omega}>0$ such that, if
\begin{equation}\label{eq:CondRegularity}
C_{t,p,\Omega}
\cdot 
\nu
\cdot
\grklam{
\nnrm{f}{B^{t-2}_p(L_p(\Omega))^d}
+
\nnrm{g}{B_p^{t-1/p}(L_{p}(\partial\Omega))^d}
}  
< 
\frac{1}{4\cdot  \nrm{L_{t,p,\Omega}}^2},
\end{equation}
then \eqref{eq:NavierStokes} has at least one solution $u\in H^{3/2}(\Omega)^d$ with corresponding pressure $\pi\in H^{1/2}(\Omega)$ which satisfy
\begin{equation}\label{eq:BesovReg:sol}
u\in B^{s_1}_{\tau_1}(L_{\tau_1}(\Omega))^d, \quad \frac{1}{\tau_1}=\frac{s_1}{d}+\frac{1}{2},\quad 0<s_1<\min\ssggklam{\frac{3}{2}\cdot\frac{d}{d-1},2},
\end{equation}
and 
\begin{equation}
\pi\in B^{s_2}_{\tau_2}(L_{\tau_2}(\Omega)),\quad \frac{1}{\tau_2}=\frac{s_2}{d}+\frac{1}{2},\quad 0<s_2< \frac{1}{2}\cdot\frac{d}{d-1},
\end{equation}
respectively. 
The pair $(u,\pi)$ is unique in $A_{1/2}
:=\{(v,q)\in X_{r,p,\Omega}\colon  
\nrm{L_{t,p,\Omega}}\cdot\nu\cdot C_{t,p,\Omega}\cdot
\nnrm{(v,q)}{X_{t,p,\Omega}}
\leq 
1/2
\}$.  
\end{theo}

\begin{bem}
\begin{enumerate}[align=right,label=\textup{(\roman*)}]
\item Note that the restriction~\eqref{eq:eps:Lip} is empty if $d=3$, i.e., the theorem holds for any arbitrary Lipschitz domain $\Omega\subseteq\R^3$ with connected boundary.
\item A close look to the proof below reveals that the constant $C=C_{t,p,\Omega}>0$ in the statement of Theorem~\ref{Th:NavierStokesDG2} can be chosen to be the product of the embedding constants of the embeddings in~\eqref{eq:embed:standard} and \eqref{eq:embed:Sobolev}.
\item The solution $u$ to \eqref{eq:NavierStokes} and the corresponding pressure $\pi$ determined in Theorem~\ref{Th:NavierStokesDG2} have $L_2$-Sobolev regularity $3/2$ and $1/2$, respectively.  To the best of our knowledge there is no result which guarantees that a solution $u$ to~\eqref{eq:NavierStokes} and a corresponding pressure term $\pi$  have a higher $L_2$-Sobolev regularity in the given setting. 
However, their Besov regularity in the scale~\eqref{eq:scale} is strictly higher than $3/2$ and $1/2$, respectively, see~\eqref{eq:BesovReg:sol} and \eqref{eq:BesovRegp}. Therefore, the usage of adaptive wavelet schemes for solving \eqref{eq:NavierStokes} is justified in the sense described in the introduction, see also Remark~\ref{rem:MitreaAndWe}.  
\end{enumerate}
\end{bem}

\begin{proof}[Proof of Theorem~\ref{Th:NavierStokesDG2}]
Fix the parameters $p$ and $t$ as well as the body force $f$ and the velocity field $g$ as required in our assumptions. We prove that the mapping
\begin{align*}
T:=T_{t,p,\Omega}^{f,g,\nu}: X_{t,p,\Omega} &\to X_{t,p,\Omega} \\
							(u,\pi)&\mapsto L_{t,p,\Omega}(f-\nu u\cdot(\nabla u),g),
\end{align*}
has a fixed point $(u,\pi)\in X_{t,p,\Omega}$ consisting of a solution $u$ to~\eqref{eq:NavierStokes} and the corresponding pressure $\pi$, both of them having the asserted properties. To this end, we first need some preparations.

First of all we check that the operator $T$ is well-defined. Since $p>2$, $1<t<2$, and $\Omega$ is bounded, the embeddings
\begin{align}\label{eq:embed:standard}
L_p(\Omega)\hookrightarrow B^{t-2}_p(L_p(\Omega))\cap L_2(\Omega)\quad\text{and}\quad
B^t_p(L_p(\Omega)) \hookrightarrow W^1(L_p(\Omega)),
\end{align}
hold, since the classical embeddings for Besov and Triebel-Lizorkin spaces, as they can be found, e.g., in \cite[Proposition~2.3.2/2]{Tri1983}, can be carried over to the case of bounded Lipschitz domains (definition via restriction) and the Sobolev spaces $W^k(L_p(\Omega))$, $k\in\bN_0$, can be described in terms of Triebel-Lizorkin spaces, see, e.g., \cite[Proposition~1.122(i)]{Tri2006}.
Moreover, since $t>d/p$, $t\notin\bN$, Sobolev's embedding theorem yields
\begin{align}\label{eq:embed:Sobolev}
B^t_p(L_p(\Omega)) =W^t(L_p(\Omega))\hookrightarrow L_\infty(\Omega),
\end{align}
see, e.g., \cite[Chapter 2.2.4]{RunSic1996}. These embeddings imply that $v\cdot (\nabla v)\in B^{t-2}_p(L_p(\Omega))^d\cap L_2(\Omega)^d$ whenever $v\in B^t_{p}(L_p(\Omega))^d$, since
\begin{align*}
\nnrm{v\cdot(\nabla v)}{B^{t-2}_p(L_p(\Omega))^d} 
+ 
\nnrm{v\cdot(\nabla v)}{L_2(\Omega)^d}
&\leq C\,
\nnrm{v\cdot(\nabla v)}{L_p(\Omega)^d}\\
&= C\,
\sgnnrm{\sum_{i=1}^d v_i\frac{\partial v}{\partial x_i}}{L_p(\Omega)^d}\\
&\leq C\,
\nnrm{v}{L_\infty(\Omega)^d} \nnrm{v}{W^1(L_p(\Omega))^d}\\
&\leq C\,
\nnrm{v}{B^t_p(L_p(\Omega))^d}^2,
\end{align*}
with a finite constant $C=:C_{t,p,\Omega}$, which is the product of the embedding constants of the embeddings above.
As a consequence, the operator
\begin{align*}
N:=N_{t,p,\Omega}^{f,g,\nu}\colon X_{t,p,\Omega} &\to Y_{t,p,\Omega}\cap \rklam{L_2(\Omega)^d\times H^1(\partial\Omega)^d}\\
(u,\pi) & \mapsto \rklam{f- \nu u\cdot (\nabla u), g}
\end{align*}
is well-defined, and, therefore, so is $T=L\circ N\colon  X_{t,p,\Omega}\to X_{t,p,\Omega}$.

Secondly, we prove the existence of a fixed point of $T$. A similar calculation as above shows that for $(u,\pi),(\tilde{u},\tilde{\pi})\in X_{t,p,\Omega}$,
\begin{align*}
\nnrm{T_{t,p,\Omega}(u,\pi) &- T_{t,p,\Omega}(\tilde{u},\tilde{\pi})}{X_{t,p,\Omega}}\\
&\leq
\nrm{L_{t,p,\Omega}}\cdot
\nu\cdot 
\nnrm{u\cdot(\nabla u)-\tilde{u}\cdot(\nabla\tilde{u})}{B^{t-2}_{p}(L_p(\Omega))^d}\\
&\leq
\nrm{L_{t,p,\Omega}}\cdot
\nu\cdot 
C_{t,p,\Omega}\cdot
\max\ggklam{\nnrm{u}{B^t_p(L_p(\Omega))^d},\nnrm{\tilde{u}}{B^t_p(L_p(\Omega))^d}}\cdot 
\nnrm{u-\tilde{u}}{B^t_p(L_p(\Omega))^d}.
\end{align*}
Thus, if we can find $\lambda<1$ such that
\[
(u,\pi)\in A_{\lambda}
:=
\ggklam{ 
(v,q)\in X_{r,p,\Omega}\colon  
\nrm{L_{t,p,\Omega}}\cdot
\nu\cdot 
C_{t,p,\Omega}\cdot
\nnrm{(v,q)}{X_{t,p,\Omega}}
\leq 
\lambda
}
\]
implies $T(u,\pi)\in A_{\lambda}$, then $T$ is a contraction on a the closed ball $A_\lambda\subseteq X$, so that we can obtain the fixed point we are seeking for by applying Banach's fixed point theorem. Assume $(u,\pi)\in A_{\lambda}$ for some $\lambda\in(0,1)$. Then,
\begin{align*}
\nnrm{T_{t,p,\Omega}&(u,\pi)}{X_{t,p,\Omega}}\\
&\leq
\nrm{L_{t,p,\Omega}}\,
\grklam{
\nnrm{f}{B^{t-2}_p(L_p(\Omega))^d}
+
\nnrm{g}{B_p^{t-1/p}(L_{p}(\partial\Omega))^d}
+
\nu \,
\nnrm{u\cdot(\nabla u)}{B^{t-2}_{p}(L_p(\Omega))^d}
}\\
&\leq
\nrm{L_{t,p,\Omega}}\,
\grklam{
\nnrm{f}{B^{t-2}_p(L_p(\Omega))^d}
+
\nnrm{g}{B_p^{t-1/p}(L_{p}(\partial\Omega))^d}
+
\nu \, C_{t,p,\Omega}\,
\nnrm{u}{B^{t}_{p}(L_p(\Omega))^d}^2
}\\ 
&\leq
\nrm{L_{t,p,\Omega}}\,
\sgrklam{
\nnrm{f}{B^{t-2}_p(L_p(\Omega))^d}
+
\nnrm{g}{B_p^{t-1/p}(L_{p}(\partial\Omega))^d}
+
\frac{\lambda^2}{\nu\,C_{t,p,\Omega}\,\nrm{L_{t,p,\Omega}}^2}
}.
\end{align*}
Moreover, 
\begin{align*}
\nrm{L_{t,p,\Omega}}\,
\sgrklam{
\nnrm{f}{B^{t-2}_p(L_p(\Omega))^d}
+
\nnrm{g}{B_p^{t-1/p}(L_{p}(\partial\Omega))^d}
+
\frac{\lambda^2}{\nu\,C_{t,p,\Omega}\,\nrm{L_{t,p,\Omega}}^2}
}
&\leq \frac{\lambda}{\nu\,C_{t,p,\Omega}\,\nrm{L_{t,p,\Omega}}}
\end{align*}
if, and only if,
\begin{align*}
\nu\, C_{t,p,\Omega}\,\nrm{L_{t,p,\Omega}}^2\,
\grklam{
\nnrm{f}{B^{t-2}_p(L_p(\Omega))^d}
+
\nnrm{g}{B_p^{t-1/p}(L_{p}(\partial\Omega))^d}
}
\leq -\lambda^2+\lambda.
\end{align*}
The right hand side is positive on $(0,1)$ and attains its maximum $1/4$ at $\lambda=1/2$. 
Thus, $\lambda_0=1/2$ does the job, so that, if~\eqref{eq:CondRegularity} is assumed to hold,  then we have a fixed point $(u,\pi)\in A_{1/2} \subseteq X_{t,p,\Omega}$.

Finally, we note that 
\[
X_{t,p,\Omega}\hookrightarrow H^{1}(\Omega)^d\times \rklam{L_2(\Omega)/\R_\Omega},
\] 
due to the embeddings discussed at the beginning of the proof.
Moreover, $u\cdot(\nabla u)\in L_2(\Omega)$, as shown above. 
Therefore, any fixed point $(u,\pi) \in X_{t,p,\Omega}$ consists, by definition, of a solution $u\in H^1(\Omega)^d$ to the Stokes problem~\eqref{St} with body force $(f-\nu u\cdot(\nabla u))\in L_2(\Omega)$ (instead of $f$) and prescribed velocity field $g\in H^1(\partial\Omega)^d$ together with its corresponding pressure $\pi\in L_2(\Omega)$. Consequently, $u\in H^{3/2}(\Omega)^d$ and $\pi\in H^{1/2}(\Omega)$ due to Proposition~\ref{Prop:SobRegSt} and they have the asserted Besov regularity in the scale~\eqref{eq:scale} due to Theorem~\ref{Th:BesovStokes}.
\end{proof}

\addcontentsline{toc}{section}{Literaturverzeichnis}

\end{document}